\documentclass[11pt]{amsart}  

 \usepackage[dvips]{epsfig}
\usepackage{graphicx}
\usepackage{latexsym}
\usepackage{amsfonts,amsmath,amssymb}
\newtheorem{corollary}{Corollary}

\newtheorem{example}{Example}
\newtheorem{theorem}{Theorem}

\newtheorem{definition}{Definition}

\begin{document}

\author[M. B\'ona]{Mikl\'os  B\'ona}
\title[Surprising symmetries in 132-avoiding permutations]{Surprising
symmetries in 132-avoiding permutations}
\address{\rm M. B\'ona, Department of Mathematics, 
University of Florida,
358 Little Hall, 
PO Box 118105, 
Gainesville, FL 32611--8105 (USA)
}

\date{}

\begin{abstract} 
We prove that the total number $S_{n,132}(q)$ of copies of 
the pattern $q$ in all 132-avoiding permutations
of length $n$ is the same
for $q=231$, $q=312$, or $q=213$. We provide a combinatorial
proof for this unexpected threefold symmetry. We then 
significantly generalize
this result to show an exponential number
of different pairs of patterns $q$ and $q'$
of length $k$ for which $S_{n,132}(q)=S_{n,132}(q')$ and the equality
 is non-trivial.
\end{abstract}

\maketitle

\centerline{{\em Dedicated to the memory of Herb Wilf}}

\section{Introduction}
\subsection{Background and Definitions}
Let $q=q_1 q_2\ldots q_k$ be a permutation in the
symmetric group $S_k$.  
We say that the permutation
$p=p_1 p_2 \ldots p_n\in
 S_n$ {\it contains a $q$-pattern\/}
 if and only if there is a subsequence
$p_{i_1}p_{i_2}\ldots p_{i_k}$  of $p$ whose elements are in the same
relative order as those in $q$, that is,
$$
p_{i_t}<p_{i_u} \mbox{ if and only if } q_t<q_u
$$
whenever $1\leq t,u\leq k$. If $p$ does not contain $q$, then we say that
$p$ {\em avoids} $q$.  For instance, 214653 contains 231 (consider the 
third, fourth, and sixth entries), but avoids 4321.
See Chapter 14 of \cite{bona}
for an introduction to pattern avoiding permutations, and Chapters 4 and 5
of \cite{combperm} for a somewhat more detailed treatment.

It is straightforward to compute, using the linear property of expectation,
 that the average number of $q$-patterns in a randomly selected permutation
of length $n$ is $\frac{1}{k!}{n\choose k}$, where $k$ is the length of $q$.

Joshua Cooper \cite{cooper}
 has raised the following interesting family of questions. 
Let $r$ be a given permutation pattern. What can be said about the
average number of occurrences of $q$ in a randomly selected $r$-avoiding
permutation of a given length?
 Equivalently, can we determine the {\em total number}
$S_{n,r}(q)$ of all $q$-patterns in all $r$-avoiding permutations of
length $n$? 

\subsection{Earlier Results}
In \cite{occurrences}, present author found formulae for the generating
 functions of 
the sequence $S_{132,n}(q)$ for the cases of monotone $q$, that is,
for $q=12\cdots k$ and $q=k(k-1)\cdots 1$, for any $k$. He also 
proved that if $n$ is large enough, then for any fixed $k$, 
  among all patterns $q$ of length $k$, it is
the monotone decreasing pattern that maximizes $S_{132,n}(q)$ and it is 
the monotone increasing pattern that minimizes $S_{132,n}(q)$.

\subsection{The Outline of our Paper}
In this paper, we first present a computational proof of the
surprising fact that for all $n$, the equalities
\begin{equation} \label{triple}
S_{132,n}(231)=S_{132,n}(312)=S_{132,n}(213)
\end{equation}
hold. The first equality is trivial, since taking the inverse 
of a 132-avoiding permutation keeps that permutation 132-avoiding,
and turns 231-patterns into 312-patterns. However, the second
equality is non-trivial. (The reverse or complement of 
a 132-avoiding permutation is not necessarily 132-avoiding.)
 In particular, if $a(p)$ denotes the number of
213-copies in $p$, and $b(p)$ denotes the number of 231-copies in $p$, then
the statistics $a(p)$ and $b(p)$ are {\em not} equidistributed over the
set of all 132-avoiding permutations of length $n$, but their
average values are equal over that set. 

In other words, we will prove that a randomly selected
non-monotonic pattern of length three in a 132-avoiding
permutation is equally likely to be a 231-pattern, a 
312-pattern, or a 213-pattern. It is well-known (see Chapter 14 of \cite{bona})
that 132-avoiding permutations of length $n$ are
counted by the Catalan numbers $c_n={2n\choose n}/(n+1)$, and as such,
they are one of more than 150 distinct kinds of objects counted by those
numbers. 
 However, we do not know of any other example when a natural statistic
on objects counted by Catalan numbers shows a similar threefold symmetry. 

In the next part of the paper we provide a bijective proof
of (\ref{triple}). Finally, we will significantly generalize this
result by showing more than $c_{h-2}$
 pairs of patterns $q$ and $q'$ of length $h$ that 
behave as 213 and 231, that is, for which $S_{n,132}(q)=
S_{n,132}(q')$, and the equality is non-trivial.

\section{Arguments Using Generating Functions}
Let $d_n$ be the total number of inversions (in other words, copies
of the pattern 21) in 
all 132-avoiding $n$-permutations. It is proved in \cite{occurrences} that 
\begin{equation} 
\label{inversions} D(x)= \sum_{n\geq 1}d_nx^n=\frac{x}{1-4x} \cdot
\left (\frac{1}{\sqrt{1-4x}} - \frac{1-\sqrt{1-4x}}{2x} \right).
\end{equation}

\subsection{Counting Copies of 213}
Let $a_n$ be the total number of all 213-patterns in all 132-avoiding 
permutations of length $n$. Clearly, then $a_0=a_1=a_2=0$.

There are three ways that a 132-avoiding permutation $p$ of length $n$
 can contain
a 213-pattern $q$. Either $q$ is entirely on the left of the entry $n$, or
$q$ is entirely on the right of $n$, or $q$ ends in $n$. 

 For $n\geq 3$, this leads to the recurrence relation
\[a_n= \sum_{i=1}^n a_{i-1}c_{n-i} + \sum_{i=1}^{n} c_{n-1}a_{n-i} + \sum_{i=3}^n
d_{i-1} c_{n-i}.\]

Let $A(x)$ (resp. $C(x)$) be the ordinary generating function for the
sequence of the numbers $a_n$ (resp. $c_n$).
Then the last displayed formula 
 yields the functional equation
\[A(x)= 2xA(x)C(x)+xD(x)C(x) ,\]
which is equivalent to 
\begin{equation}
\label{explicitA} A(x)=\frac{xD(x)C(x)}{1-2xC(x)}=\frac{x}{2(1-4x)^2}
+\frac{x-1}{2(1-4x)^{3/2}} + \frac{1}{2(1-4x)}.\end{equation}
From here, we get that if $n\geq 3$, then  
\[a_n=\frac{n}{2}4^{n-1}+\frac{1}{2}4^n-(2n+1){2n-1\choose n-1}+
(2n-1){2n-3\choose n-2}, \]
which simplifies to 
\begin{equation}
\label{exactfora} a_n=(n+4)\cdot 2^{2n-3}-(2n+1){2n-1\choose n-1}+
(2n-1){2n-3\choose n-2}. \end{equation}

\subsection{Counting Copies of 231}
Let $h_n$ be the total number of all non-inversions (in other words, 
copies of the pattern 12) in all
132-avoiding permutations of length $n$. It is proved in \cite{occurrences} that
\begin{equation} \label{non-inversions}
H(x)=\sum_{n\geq 0}h_nx^n=\frac{1}{2(1-4x)} + \frac{1}{2x}
-\frac{1-x}{2x\sqrt{1-4x}}.
\end{equation}

Let $b_n$ be the total number of all 231-copies in all 
132-avoiding permutations of length $n$, and let $B(x)=\sum_{n\geq 0}b_nx^n$. 

Let \begin{equation}
\label{explicitZ} Z(x)=\sum_{n\geq 0}nc_nx^n=
\sum_{n\geq 0}{2n\choose n}\frac{n}{n+1}x^n=
\frac{1}{\sqrt{1-4x} }- \frac{1-\sqrt{1-4x}}{2x}.\end{equation}
Note that
$Z(x)$ is the generating function for the number of entries (which are
copies of the pattern 1)
in all 132-avoiding $n$-permutations. 

If $p$ is a 132-avoiding $n$-permutation, and $q$ is a 231-pattern contained
in $p$, then either $q$ is entirely on the left of the entry $n$, or $q$
is entirely on the right of the entry $n$, or the entry $n$ is the largest
entry of $q$, or the first and second entries of $q$ form a 12-pattern on 
the left of $n$, while the third entry of $q$ is on the right of $n$.

For $n\geq 3$, this leads to the recurrence relation
\[b_n=\sum_{i=1}^n b_{i-1}c_{n-i} + \sum_{i=1}^{n} c_{n-1}b_{n-i} +
\sum_{i=2}^{n-1}(i-1)(n-i)c_{i-1}c_{n-i} + \sum_{i=3}^{n-1} h_{i-1}c_{n-i}(n-i).\]
In terms of generating functions, this yields
\[B(x)=2xB(x)C(x) + xZ^2(x)+xH(x)Z(x),\]
\[B(x)=\frac{ xZ^2(x)+xH(x)Z(x)}{1-2xC(x)}=
\frac{ xZ^2(x)+xH(x)Z(x)}{\sqrt{1-4x}}.\]
Given the explicit formulae (\ref{non-inversions}) and (\ref{explicitZ})
for $H(x)$ and $Z(x)$, the last displayed equation yields
the formula 
\begin{equation} \label{explicitB}
 B(x)=\frac{xD(x)C(x)}{1-2xC(x)}=\frac{x}{2(1-4x)^2}
+\frac{x-1}{2(1-4x)^{3/2}} + \frac{1}{2(1-4x)}.
\end{equation}

The proof of the main result of this section is now immediate.
\begin{theorem}
For all positive integers $n$, the equalities
\[
S_{132,n}(231)=S_{132,n}(312)=S_{132,n}(213)\]
hold.
\end{theorem}

\begin{proof} As we mentioned in the Introduction, the first
equality is trivially true since there is a natural bijection
between the 231-copies of the 132-avoiding permutation $p$
and the 312-copies of the 132-avoiding permutation $p^{-1}$. Indeed, let 
$p=p_1p_2\cdots p_n$ be a 132-avoiding permutation, and let
$1\leq i<j<k\leq n$.
Then $p_ip_jp_k$ is a 231-copy in $p$ if and only if 
$ijk$ is a 312-copy in $p^{-1}$. 

The equality $S_{n,132}(231)=S_{n,132}(213)$ holds since we have seen
in formulae (\ref{explicitA}) and (\ref{explicitB}) that the two 
sides of this equality have identical generating functions. 
\end{proof}

\section{A Bijective Proof} 
In this section we provide a bijective proof for  the surprising
identity $S_{n,132}(213)=S_{n,132}(231)$. 

\subsection{Binary Plane Trees}
In our proof, we will identify a 132-avoiding permutation $p$ with its
{\em binary plane tree} $T(p)$ using a very well-known bijection. 
We will briefly describe this bijection now. For more details,
the reader may consult Chapter 14 of \cite{bona}.
The tree $T(p)$ will be a  binary plane tree, that is,
a rooted unlabeled
 tree in which each vertex has at most two children, and each child
is a left child or a right child of its parent, even if it is the only
child of its parent.  

The root of $T(p)$ corresponds to the entry $n$ of $p$, the left subtree
of the root corresponds to the string of entries of $p$ on the left of
$n$, and the right subtree of the root corresponds to the string of
entries of $p$ on the right of $n$. Both subtrees are constructed
recursively, by the same rule.  Note that since $p$ is 132-avoiding,
the position of the entry $n$ of $p$ determines the set of entries
that are on the left (resp. on the right) of $n$. In fact, if $n$
 is in the $i$th position, 
the set of entries on the left of $n$ must be $\{n-i+1,n-i+2,\cdots ,n-1\}$,
 and the set of
entries on the right of $n$ must be $\{1,2,\cdots ,n-i\}$.

We point out that in the process of constructing $T(p)$, 
each vertex of $T(p)$ is associated to an entry of
$p$.  Indeed, each vertex is added to $T(p)$ as the root of a subtree $S$,
 and so
each vertex is associated to the entry that is the largest among the entries
that belong to $S$.
However, it is important to point out that
$T(p)$ is an {\em unlabeled tree} since the way in which the entries
of $p$ correspond to the vertices of $T(p)$ is completely determined
 by the unlabeled tree $T(p)$ as long as $p$ is 132-avoiding. 

See Figure \ref{binplane} for an illustration.

\begin{figure}[ht]
 \begin{center}
  \epsfig{file=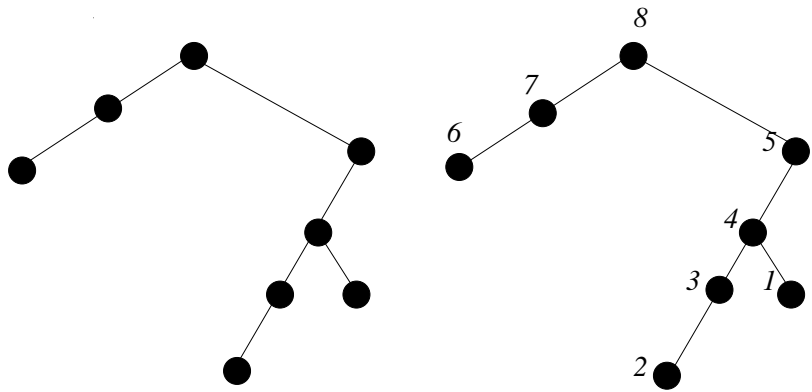}
  \label{binplane}
\caption{The tree $T(p)$ for $p=67823415$, and the entries of $p$ 
associated to the vertices of $T(p)$.}
 \end{center}
\end{figure}

Note that in order to get $p$ from $T(p)$, it suffices to read 
the vertices of $T(p)$ {\em in-order}, that is, by first reading 
the left subtree of the root, then the right subtree of the root, and
then the right subtree of the root. The respective subtrees are read
recursively, by this same rule. Therefore, it is meaningful to talk
about the first, second, etc, last vertex of $T(p)$, since that means
the first, second, etc, last vertex of $T(p)$ in the {\em in-order}
reading. 

A {\em left descendant} (resp. {\em right descendant} of a vertex $x$
in a binary plane tree is a vertex in the left (resp. right) subtree 
of $x$. The left (resp. right) subtree of $x$ does {\em not} contain $x$
itself. 

It is straightforward to see that $p_ip_j$ is a 12-pattern in $p$
if and only if $p_i$ is a left-descendant of $p_j$ in $T(p)$. On the other
hand, $p_jp_i$ is a 21-pattern in $p$ if and only if either $p_i$ is a
right descendant of $p_j$ in $T(p)$ or there is a vertex $x$ in $T(p)$ so
that $p_j$ is a left descendant of $x$ and $p_i$ is a right descendant of $x$.
In the previous section we gave an exhaustive list of the ways in which
213-patterns and 231-patterns can occur in a 132-avoiding permutation.
The reader is invited to translate that list into the language of 
binary plane trees.

\subsection{Our Bijection}
Let $p$ be a 132-avoiding $n$-permutation, and let $Q$ be an occurrence 
of the pattern 213 in $p$. Let  $Q_2,Q_1,Q_3$ be the three
vertices of $T(p)$ that correspond to  $Q$, going left to right.
 Let us
color these three entries black. There are then two possibilities. 
\begin{enumerate}
\item Either $Q_1$ is a right descendant of $Q_2$ and $Q_2$ is a left descendant
of $Q_3$, or
\item there exists a lowest left descendant $Q_x$ of $Q_3$ so that $Q_2$ is
 a left descendant
of $Q_x$ and $Q_1$ is a right descendant of $Q_x$.
\end{enumerate}

Let $A_n$ be the set of all binary plane trees on $n$ vertices in which three
vertices forming a 213-pattern are colored black. Let $B_n$ be the set
of all binary plane trees on $n$ vertices in which three vertices forming
a 231-pattern are colored black. 

Now we are going to define a map $f:A_n\rightarrow B_n$. We will then
prove that $f$ is a bijection. 
The map $f$ will be defined differently in the two cases described above.  
\begin{itemize}
\item {\em Case 1.} If $T\in A_n$ is in the first case,  then let 
$f(T)$ be the pair obtained by
interchanging the right subtree of $Q_2$ and the right subtree of $Q_3$.
Keep all three black vertices $Q_i$ black, even as $Q_1$ gets moved.  

See Figure \ref{firstmove} for an illustration.

\begin{figure}[ht]
 \begin{center}
  \epsfig{file=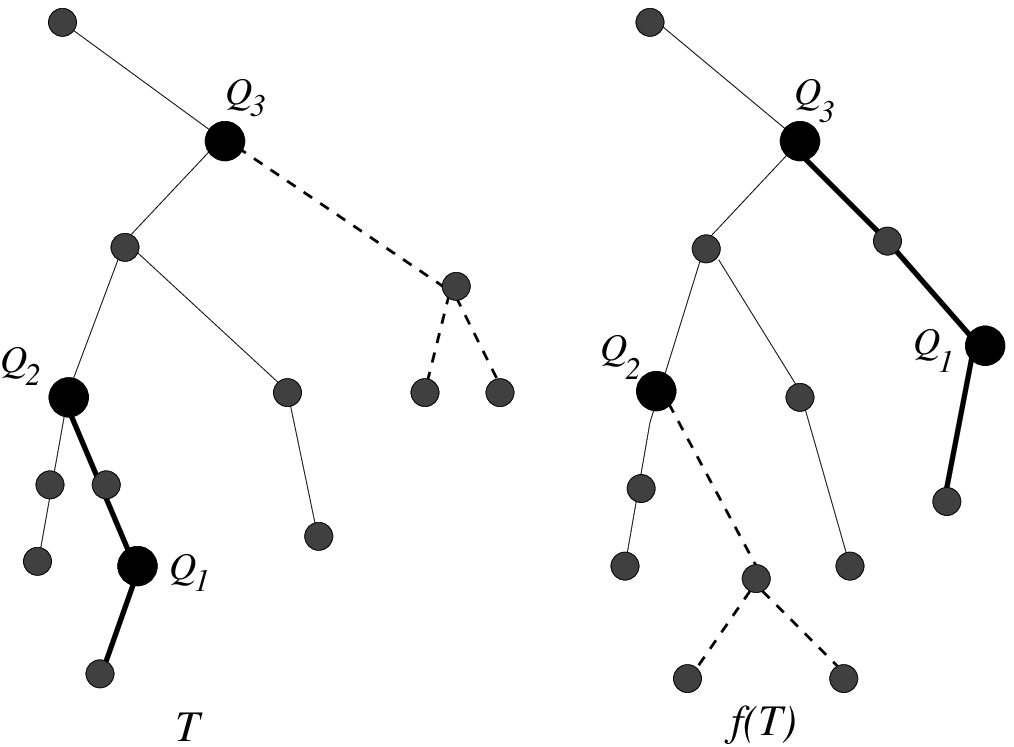}
  \label{firstmove}
\caption{Interchanging the right subtrees of $Q_2$ and $Q_3$.}
 \end{center}
\end{figure}

Note that in $f(T)$, in the set of black vertices, there is one that is
an ancestor of the other two, namely $Q_3$.

\item {\em Case 2.} If $T\in A_n$  is in the second case, then let $f(T)$ be the
tree obtained by interchanging the right subtrees of the
vertices $Q_x$ and $Q_3$, 
and coloring $Q_2$, $Q_x$ and $Q_1$ black. See Figure \ref{secondmove}
for an illustration.

\begin{figure}[ht]
 \begin{center}
  \epsfig{file=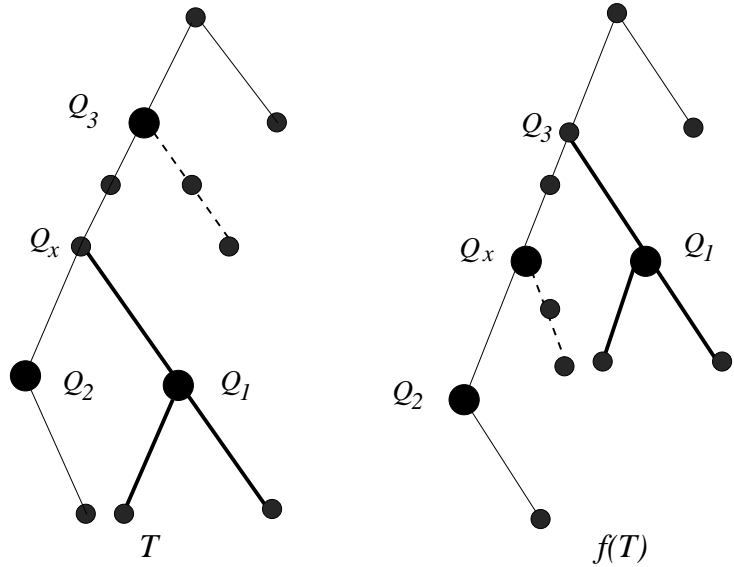}
  \label{secondmove}
\caption{Interchanging the right subtrees of $Q_x$ and $Q_3$.}
 \end{center}
\end{figure}

Note that in $f(T)$, there is no black vertex that is an ancestor of the
other two black vertices. Also note that in $f(T)$, the 
lowest common ancestor of $Q_x$ and $Q_1$ is $Q_3$.
\end{itemize}

It is a direct consequence of our definitions that if $T\in A_n$, then
 $f(T)=B_n$. 
Now we are in a position to prove the main result of this section. 

\begin{theorem}  \label{bijective}
The map $f:A_n\rightarrow B_n$ defined above is a bijection.
\end{theorem}

\begin{proof}
Let $U\in B_n$. We will show that there is exactly one $T\in A_n$ so that
$f(T)=U$ holds. This will show that $f$ has an inverse, proving that 
$f$ is a bijection.

By definition, three nodes of $U$ are colored black, and the entries
of the permutation corresponding to $U$ form a 231-pattern.
Let $K_2$, $K_3$, and $K_1$ denote these three vertices, from left to right.
There are two possibilities for the location of the $K_i$ relative to each
other. We will show that in both cases, $U$ has a unique preimage under $f$,
essentially because swapping two subtrees is an involution. 
\begin{enumerate} 
\item If $K_3$ is an ancestor of both other black
vertices,  then $f(T)=U$ implies that $T$ belongs to Case 1.
In this case,  the unique $T\in A_n$ satisfying $f(T)=U$ is
 obtained
by swapping the right subtrees of $K_3$ and $K_2$, and keeping all three black
vertices black, even if $K_1$ got moved. 

\item If $K_3$ is not an ancestor of both other black vertices and then
 $f(T)=U$ implies that 
$T$ belongs to Case 2. 
In this case,  let $K_x$ be the smallest common ancestor
of $U_3$ and $U_1$. Then the unique  $T\in A_n$ satisfying $f(T)=U$ is obtained
by swapping the right subtrees of $K_3$ and $K_x$, and coloring $K_x$ black
instead of $K_3$, while keeping $K_1$ and $K_2$ black. 
\end{enumerate}  
This completes the proof.
\end{proof}

\section{A Generalization}
In this section, we will significantly generalize the result of the
previous section. The key observation is that in the proof of Theorem 
\ref{bijective}, the {\em left} subtrees of $Q_1$ and $Q_2$ never changed.

In order to state our result, we announce the following
definitions. 

\begin{definition}
Let $q$ be a pattern of length $k$ and let $t$ be a pattern of length $m$.
Then $q\oplus t$ is the pattern of length $k+m$ defined by
\[ (q\oplus t)_i= \left\{ \begin{array}{l@{\ }l}
q_i  \hbox{  if $i\leq k$},\\
\\
 t_{i-k} +k \hbox{  if $i>k$. }
\end{array}\right.
\]
\end{definition}

In other words, $q\oplus t$ is the concatenation of $q$ and $t$ so that 
all entries of $t$ are increased by the size of $q$.

\begin{example}
If $q=3142$ and $t=132$, then $q\oplus t=3142576$.
\end{example}

\begin{definition}
Let $q$ be a pattern of length $k$ and let $t$ be a pattern of length $m$.
Then $q\ominus t$ is the pattern of length $k+m$ defined by
\[ (q\ominus t)_i= \left\{ \begin{array}{l@{\ }l}
q_i + m \hbox{  if $i\leq k$},\\
\\
 t_{i-k} \hbox{  if $i>k$. }
\end{array}\right.
\]
\end{definition}

In other words, $q\ominus t$ is the concatenation of $q$ and $t$ so that 
all entries of $q$ are increased by the size of $t$.

\begin{example}
If $q=3142$ and $t=132$, then $q\ominus t=6475132$.
\end{example}

Now we are ready to state and prove the most general result of this paper.

\begin{theorem} \label{general}
Let $q$ and $t$  be any non-empty patterns  that end in their largest entry.
Let $i_u$ denote the increasing pattern $12\cdots u$.
Then for all positive integers $n$, we have
\[S_{n,132}((q\ominus t)\oplus i_u) = S_{n,132}((q\oplus i_u)\ominus t),\]
where 1 denotes the pattern consisting of one entry.
\end{theorem}

In particular, the result of the previous section  is the special case
of Theorem \ref{general} in which $q=t=i_u=1$ (the one-entry pattern 1). 

\begin{example} If $q=3124$,  $t=213$, and $u=2$,
 then Theorem \ref{general} says
that \[S_{n,132}(645721389)=S_{n,132}(645789213).\]
\end{example}

\begin{proof} (of Theorem \ref{general}) 
Note that we can assume that $q$ and $t$ are both 132-avoiding, since
otherwise the statement of Theorem \ref{general} is trivially true as both
sides are equal to 0.

Let $k$ denote the length of $q$, let $m$ denote the length of $t$.
 Similarly to the proof of Theorem \ref{bijective}, 
 let $A_n$ be the set of all binary plane trees on $n$ vertices in which
$h$ vertices forming a $((q\ominus t)\oplus i_u)$-pattern are colored black, 
and let $B_n$
be the set of all binary plane trees on $n$ vertices in which $h$ vertices
forming a $((q\oplus i_u)\ominus t)$-pattern are colored black. 

Let $T\in A_n$. Let $Q_b$ be the $k$th black vertex of $T$ in the
in-order reading, let  $Q_a$ be 
the $(k+m)$th black vertex of $T$, and let  $Q_c$ be the rightmost  
black vertex of $T$. We are now going to construct a bijection 
$F:A_n\rightarrow B_n$. The construction is analogous to the one that
we saw before Theorem \ref{bijective}

\begin{enumerate}
\item If $Q_a$ is a right descendant of $Q_b$, then let $F(T)$ be
the tree obtained from $T$ by swapping the right subtree of $Q_b$
and the right subtree of $Q_c$. Note that in $F(T)$, the black vertices
form a  $(q\ominus t)\oplus i_u$-pattern, and that $Q_c$ is an ancestor
of all other black vertices  in $F(T)$.
See Figure \ref{firstg} for an illustration.

\begin{figure}[ht]
 \begin{center}
  \epsfig{file=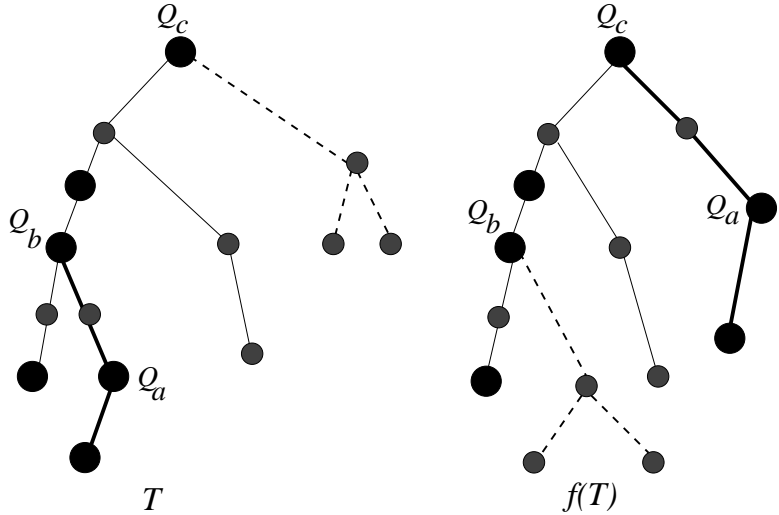}
  \label{firstg}
\caption{Interchanging the right subtrees of $Q_b$ and $Q_c$, and
turning a copy of 341256 into a copy of 345612.}
 \end{center}
\end{figure}

\item  Otherwise,  there exists a  lowest vertex $Q_x\in T$ so that $Q_b$
 is a left descendant
of $Q_x$ and $Q_a$ is a right descendant of $Q_x$. Note that in this
case, it follows that $Q_x$ is not black. Now let $F(T)$ be the
tree obtained from $T$ by swapping the right subtree of $Q_x$ and the right
subtree of $Q_c$, and by coloring $Q_x$ black, instead of $Q_c$. Note that
again, in $F(T)$, the black vertices
form a  $(q\ominus t)\oplus i_u$-pattern. Also note that there is no black 
vertex
in $F(T)$ that would be an ancestor of all other black vertices.
See Figure \ref{secondg} for an illustration.

\begin{figure}[ht]
 \begin{center}
  \epsfig{file=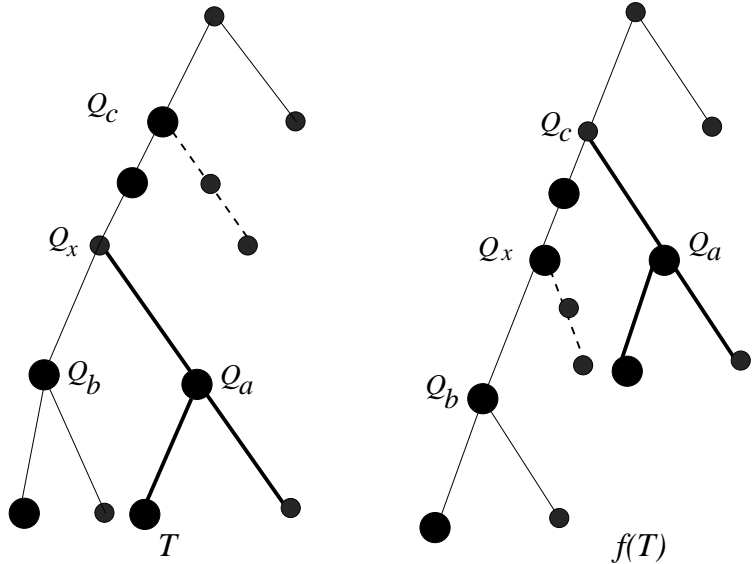}
  \label{secondg}
\caption{Interchanging the right subtrees of $Q_x$ and $Q_c$, 
coloring $Q_x$ black instead of $Q_c$, and
turning a copy of 341256 into a copy of 345612.}
 \end{center}
\end{figure}

\end{enumerate}

It is straightforward to show that $F:A_n \rightarrow B_n$ is a bijection.
Indeed, let $U\in B_n$. If there is a black vertex in $U$ that is an ancestor
of all other black vertices, then $U$ could only be obtained by the first
rule, otherwise $U$ could only be obtained by the second rule. 
The unique preimage $F^{-1}(U)$ is then obtained by swapping the appropriate
right subtrees. In the first case, swap the right subtrees of $U_b$ and
$U_c$, where $U_b$ is the $k$th and $U_c$ is the $(k+u)$th black vertex
of $U$ in the in-order reading.
 In the second case, let $U_x$ be the $(k+1)$st black vertex
of $U$, let $U_a$ be the last black vertex of $U$, 
and let $U_c$ be the lowest common ancestor of $U_x$ and $U_a$. Then
the unique preimage $F^{-1}(U)$ is obtained by swapping the right 
subtrees of $U_x$ and $U_c$, and coloring $U_c$ black instead of $U_x$.
\end{proof}

Note that by transitivity, Theorem \ref{general} implies the following.
\begin{corollary}
Let $q$, $t$, and $i_u$ be as in Theorem \ref{general}, and let 
$1\leq v<u$. Then we have
\[S_{n,132}((q\oplus i_v)\ominus t)\oplus i_{u-v})=S_{n,132}((q\ominus t)\oplus i_u).\]
\end{corollary}

\begin{proof}
Theorem \ref{general} shows that both
 sides are equal to $S_{n,132}((q\oplus i_u)\ominus t)$.
\end{proof}

\section{Further Directions}
Formula (\ref{exactfora}) implies that $S_{n,132}(213)\sim C_1 4^n n$, while
the  generating functions computed in \cite{occurrences} imply that
$S_{n,132}(321)\sim C_2 4^n n^{3/2}$ and $S_{n,132}(123)\sim C_3 4^n n^{1/2}$,
where the $C_i$ are positive constants.
So occurrences of non-monotone patterns of length three
 are infinitely rare compared to 
occurrences of 321, and infinitely frequent compared to occurrences
of 123; the frequency of non-monotone patterns is halfway between the two 
extremes. 

While precise formulae like the ones given in earlier sections of 
this paper may not be obtainable
for longer patterns, comparative results as the ones described in the previous
paragraph may be possible to prove even for such patterns. 

If we set $u=1$ and  $h=k+m+1$, then Theorem \ref{general}
 provides
$\sum_{i=2}^{h-1}c_{i-2}c_{h-i-1}=c_{h-2}$ non-trivial examples  
 of two patterns
$s$ and $s'$ for which $S_{n,132}(s)=S_{n,132}(s')$ for all $n$.
Other choices of $u$ provide additional such pairs. However, 
it seems that there are other pairs of patterns whose total number
of copies in all 132-avoiding permutations agree. We hope to discuss
such pairs in an upcoming paper. 
 
 Are there any other such
pairs? Are there any such pairs when 132 is replaced by another pattern $r$?
Are there any patterns $r$ and $r'$ for which $S_{n,r}(u)=S_{n,r'}(u')$ for all
$n$ and
the equality is non-trivial?  

Finally, how do the permutation statistics studied in this paper 
translate to the other 150 families of objects counted by the 
Catalan numbers listed in \cite{stanley}?


\begin{thebibliography}{99}
\bibitem{bona} M. B\'ona, A Walk Through Combinatorics, 3rd edition,
World Scientific, 2011.
\bibitem{combperm}  M. B\'ona, Combinatorics of Permutations, 
CRC Press, 2004.
\bibitem{rules}  M. B\'ona, Where the monotone pattern (mostly)
rules,  {\em Discrete Math}.  {\bf 308}  (2008),  no. 23, 5782--5788.
\bibitem{occurrences}  M. B\'ona, The absence of a pattern and the occurrences
 of another. {\em Discrete Math. Theor. Comput. Sci.} {\bf 12} (2010), no. 2,
 89--102.
\bibitem{cooper}
 J. Cooper, Combinatorial Problems I like, {\em internet resource}.
\bibitem{patternsin} N. Ruskuc, S. Linton, V. Vatter, editors, 
Patterns in Permutations, Cambridge University Press, 2010.
\bibitem{stanley} R. Stanley, Enumerative Combinatorics, Volume 2, Cambridge
University Press, 1997.
\end{thebibliography}
\end{document}